\documentclass[11pt]{amsart}
\usepackage{graphicx, color}
\usepackage{amscd}
\usepackage{amsmath}
\usepackage{amsfonts}
\usepackage{amssymb}
\usepackage{mathrsfs}
\newtheorem{theorem}{Theorem}
\newtheorem{lemma}[theorem]{Lemma}
\newtheorem{ex}[theorem]{Example}

\newtheorem{remark}[theorem]{Remark}

\numberwithin{equation}{section}
\newcommand{\R}{\mathbb R}
\newcommand{\N}{\mathbb N}

\begin{document}

\title[Gradient--type systems]{Gradient-type systems on unbounded \\ domains of the Heisenberg group}


\author{Giovanni Molica Bisci}
\address[G. Molica Bisci]{Dipartimento di Scienze Pure e Applicate (DiSPeA),
 Universit\`{a} degli Studi di Urbino Carlo Bo, Piazza della Repubblica 13,
61029 Urbino, Italy}
\email{\tt giovanni.molicabisci@uniurb.it}

\author{Du\v{s}an D. Repov\v{s}}
\address[Du\v{s}an Repov\v{s}]{Faculty of Education, and Faculty of Mathematics and Physics,
University of Ljubljana \& 
Institute of Mathematics, Physics and Mechanics,
1000
Ljubljana, Slovenia}
\email{\tt dusan.repovs@guest.arnes.si}



\begin{abstract}
The purpose of this paper is to study the existence of weak solutions for some classes of one--parameter
subelliptic gradient--type systems involving a Sobolev--Hardy potential defined on an unbounded domain $\Omega_\psi$ of the Heisenberg group $\mathbb{H}^n=\mathbb{C}^n\times \mathbb{R}$ ($n\geq 1$) whose geometrical profile is determined by two real positive functions $\psi_1$ and $\psi_2$ that are bounded on bounded sets. The treated problems have a
variational structure and thanks to this, we are able to prove the existence of an open interval $\Lambda\subset (0,\infty)$ such that, for every parameter $\lambda\in \Lambda$, the system has at least two non--trivial symmetric weak solutions that are uniformly bounded with respect to the Sobolev $HW^{1,2}_0$--norm. Moreover, the existence is stable under certain small subcritical perturbations of the nonlinear term. The main proof, crucially based
 on the Palais principle of symmetric criticality, is obtained by developing a group--theoretical procedure on the unitary group $\mathbb{U}(n)=U(n)\times\{1\}$ and by exploiting some compactness embedding results into Lebesgue spaces, recently proved for suitable  $\mathbb{U}(n)$--invariant subspaces of the Folland--Stein space $HW^{1,2}_0(\Omega_\psi)$. A key ingredient for our variational approach is a very general min--max argument valid for sufficiently smooth functionals defined on reflexive Banach spaces.
\end{abstract}

\maketitle

\tableofcontents

\section{Introduction}\label{sec:introduzione}
The purpose of the present paper is to study the existence of weak solutions for subelliptic systems defined on unbounded domains of the Heisenberg group $\mathbb{H}^n=\mathbb{C}^n\times \mathbb{R}$ ($n\geq 1$). More precisely,
let $\psi_1,\psi_2:[0,\infty)\rightarrow \mathbb{R}$ be two functions that are bounded on bounded sets, with $\psi_1(r)<\psi_2(r)$ for every $r\geq 0$. Define
\begin{equation}\label{omega}
\Omega_\psi=\{q\in \mathbb{H}^n: q=(z,t)\,\, {\rm with}\,\, \psi_1(|z|)<t<\psi_2(|z|)\},
\end{equation}
where $|z|=\sqrt{\sum_{i=1}^{n} |z_i|^2}$, and assume that $\Omega_\psi$ contains the origin $O=(0,0)\in \mathbb{H}^n$.\par
\indent We deal here with the following singular subelliptic problem
\begin{equation}\label{p2}\left\{\begin{aligned}
&-\Delta_{\mathbb H^n}u-\nu V(q)u+u= \lambda K(q)\partial_1F(u,v)+\mu \partial_1G(q,u,v)\quad{\rm in}\,\,\Omega_\psi\\
&-\Delta_{\mathbb H^n}v-\nu V(q)v+v= \lambda K(q)\partial_2F(u,v)+\mu \partial_2G(q,u,v)\, \quad{\rm in}\,\,\Omega_\psi\\
& \,\,\,u=v=0 \,\,{\rm on }\,\, \partial\Omega_\psi,
\end{aligned}\right.\end{equation}
where $\Delta_{\mathbb H^n}$ is the Kohn--Laplace operator
 defined by
\begin{equation*}
\Delta_{\mathbb H^n}\varphi={\rm div}_H(D_{\mathbb H^n}\varphi)
\end{equation*}
along any $\varphi\in C^\infty_0(\mathbb H^n)$, with $$D_{\mathbb H^n}\varphi=(X_1\varphi,\cdots ,X_n \varphi,Y_1\varphi,\cdots, Y_n\varphi)$$  as in  Section~\ref{sec2}, and $\{X_j,Y_j\}_{j=1}^n$ is the basis of horizontal left invariant vector fields on $\mathbb H^n$, that is
$$X_j=\frac {\partial} {\partial x_j}+2y_j\frac {\partial} {\partial t}, \qquad \qquad  Y_j=\frac {\partial} {\partial y_j}-2x_j\frac {\partial} {\partial t},$$  for $j=1,\dots,n$. The {\em critical Sobolev exponent} $2^*$ in the Heisenberg group $\mathbb{H}^n$ is defined as $2^*=2Q/(Q-2)$, where $Q=2n+2$ is the homogeneous dimension of $\mathbb{H}^n$.

On the potentials $V,K:\Omega_\psi\rightarrow \R$ we assume that:
\begin{itemize}
\item[$(h_V)$] \emph{$V$ is measurable, cylindrically symmetric, i.e., $V(z,t)=V(|z|,t)$, and there exists a constant $C_V>0$ such that
    \begin{equation}\label{comparaV}
    0\leq V(z,t)\leq C_V\frac{|z|^2}{r(z,t)^4},
    \end{equation}
    for every $q=(z,t)\in \mathbb H^n$, where $r$ denotes the Heisenberg norm $r(q)=r(z,t)=(|z|^4+t^2)^{1/4} $, $z\in \mathbb{C}^n$, $t \in \mathbb R$};
    \item[$(h_K)$] $K\in L^{\infty}(\Omega_\psi)\cap L^1(\Omega_\psi)$ \textit{is a non-negative cylindrically symmetric function with} $$\inf_{q\in\Omega_0}K(q)>0,$$ \textit{for some open set} $\Omega_0\subset \Omega_\psi$.
\end{itemize}

The parameters {$\lambda,\mu>0$ and $\nu\in [0,C_V^{-1}n^2)$. Suppose also that the nonlinearity $F$ satisfies the following hypotheses:}

\begin{itemize}
\item[$(f_1)$] \textit{$F:\R^2\rightarrow\R$ is a $C^1$--function with $F(0,0)=0$ and there exists $(\eta_0,\zeta_0)\in \R^2$ such that $F(\eta_0,\zeta_0)>0$};
\item[$(f_2)$] \textit{$\displaystyle\lim_{\eta,\zeta\rightarrow 0}\frac{\partial_1 F(\eta,\zeta)}{|\eta|}=\displaystyle\lim_{\eta,\zeta\rightarrow 0}\frac{\partial_2 F(\eta,\zeta)}{|\zeta|}=0$};
    \item[$(f_3)$] \textit{there exist $\epsilon >0$ and $\alpha\in (2,2^*)$ such that
    $$
    |\partial_1F(\eta,\zeta)|\leq \epsilon \left(|\eta|+|\zeta|+|\eta|^{\alpha-1}\right),
    $$
    and
    $$
    |\partial_2F(\eta,\zeta)|\leq \epsilon \left(|\eta|+|\zeta|+|\zeta|^{\alpha-1}\right),
    $$
    for every $(\eta,\zeta)\in \R^2$};
    \item[$(f_4)$] \textit{there exist $p_F,q_F\in (0,2)$ and
    suitable real constants $\kappa_j>0$, such that
    $$
    F(\eta,\zeta)\leq \kappa_1|\eta|^{p_F}+\kappa_2|\zeta|^{q_F}+\kappa_3,
    $$
    for every $(\eta,\zeta)\in \R^2$}.
\end{itemize}

Here and in the sequel, the nonlinearities $\partial_1F$ and $\partial_2 F$ denote the partial derivatives of $F$ with respect to the first variable and the second variable, respectively.

Furthermore, for the nonlinear term $G:\Omega_\psi\times\R^2\rightarrow\R$ we assume the following conditions:

\begin{itemize}
\item[$(g_1)$] \textit{$G:\Omega_\psi\times\R^2\rightarrow\R$ is a continuous function, $(\eta,\zeta)\mapsto G(q,\eta,\zeta)$ is of class $C^1$ and $G(q,0,0)=0$ for every $q\in \Omega_\psi$};
\item[$(g_2)$] \textit{$\displaystyle\lim_{\eta,\zeta\rightarrow 0}\frac{\partial_1 G(q,\eta,\zeta)}{|\eta|}=\displaystyle\lim_{\eta,\zeta\rightarrow 0}\frac{\partial_2 G(q,\eta,\zeta)}{|\zeta|}=0$, uniformly for every $q\in \Omega_\psi$};
    \item[$(g_3)$] \textit{there exist $\kappa >0$ and $\beta\in (2,2^*)$ such that
    $$
    |\partial_1G(q,\eta,\zeta)|\leq \kappa \left(|\eta|+|\zeta|+|\eta|^{\beta-1}\right),
    $$
    and
    $$
    |\partial_2G(q,\eta,\zeta)|\leq \kappa \left(|\eta|+|\zeta|+|\zeta|^{\beta-1}\right),
    $$
    for every $q\in \Omega_\psi$ and $(\eta,\zeta)\in \R^2$};
    \item[$(g_4)$] \textit{the function $G(\cdot, \eta,\zeta)$ is cylindrically symmetric for every $(\eta,\zeta)\in \R^2$}.
\end{itemize}

The nonlinearities $\partial_1 G$ and $\partial_2 G$ denote the partial derivatives of $G$ with respect to the second variable and the third variable, respectively.

Noncompact variational problems have attracted much attention since the late seventies.
System \eqref{p2} is a reasonably useful generalization of most studied elliptic problems, with singular
potentials and subcritical nonlinearities, which naturally arise in different branches of mathematics.
More precisely, differential problems involving a subelliptic operator on an unbounded domain $\Omega$
of stratified groups have been intensively studied in recent years by several authors,
see, among others, the papers \cite{GL,M1,M2,ST,T} and references therein.

As observed in the recent paper \cite{MolicaPucci}, if the domain $\Omega$ is not bounded then the Folland--Stein space $HW^{1,2}_0(\Omega)$ need not be compactly embeddable into a Lebesgue space. This lack of compactness produces several difficulties in exploiting variational methods.
In order to recover compactness for the unbounded case, a standard hypothesis in the above cited results was the \textit{strongly asymptotically contractive} condition on $\Omega$ (see \cite{dPF} and \cite{M1} for related topics).

Now, we observe that a strongly asymptotically contractive domain $\Omega$ is geometrically thin at infinity. Following \cite{bk}, in the presence of symmetries, by replacing the contractive assumption on $\Omega$ with a technical geometrical hypothesis, we are able to treat here subelliptic gradient--type systems on the Heisenberg group $\mathbb{H}^n$, in which the domain $\Omega_\psi$ is possibly large at infinity. We also notice that if the functions $\psi_1$ and $\psi_2$ are bounded, the domain $\Omega_\psi$ is strongly asymptotically
contractive and the entire space $HW^{1,2}_{0}(\Omega_\psi)$ is compactly embedded in
$L^{q}(\Omega_\psi)$ for every $q\in (2,2^*)$.
We refer to~\cite{bk,GL,M2,ST} for further details.

In the main result of the present paper (see Theorem \ref{th1}) we prove the existence of an open interval $\Lambda\subset (0,\infty)$ such that, for every parameter $\lambda\in \Lambda$, the system \eqref{p2} has at least two symmetric weak solutions that are uniformly bounded with respect to the Sobolev $HW^{1,2}_0$--norm. This existence result is stable in the presence of a small perturbation term $G$ for which the structural conditions $(g_1)$--$(g_4)$ are satisfied.

 In order to prove Theorem \ref{th1}, we
find critical points of the energy functional associated with problem \eqref{p2} by means
of a mini-max theorem and the well--known Palais principle of
symmetric criticality (see, respectively, Theorems \ref{teorricceri} and \ref{lem3}).
More precisely, our strategy is to use the topological unitary group $\mathbb{U}(n)=U(n)\times \{1\}$. Indeed, this group acts continuously on $HW^{1,2}_0(\Omega_\psi)$ by
$$
(\widehat{\tau}\sharp u)(q)=u({\tau}^{-1}z,t)\quad\mbox{for  all } q=(z,t)\in\mathbb H^n,
$$
and the $T$--invariant closed subspace $HW^{1,2}_{0,T}(\Omega_\psi)$ associated to the subgroup
$$
T=U(n_1)\times...\times U(n_\ell)\times\{1\},\quad n=\sum_{i=1}^{\ell}n_i,\quad\mbox{with}\quad n_i\geq 1\,\,\mbox{and}\,\,\ell\geq 1,
$$
 is compactly embedded in the Lebesgue space $L^q(\Omega_\psi)$, for every $q\in (2,2^*)$, as proved in \cite{bk}; see Lemmas \ref{lem1} and \ref{lem2}. A similar argument works for strip--like domains $\Omega=\omega\times\R^{n-m}$, where $\omega\subset \R^m$ is bounded and $n-m\geq 2$, yielding the space of cylindrically symmetric functions on $H^1_0(\Omega)$ via the group $T=id_{\R^m}\times O(n-m)$ (see \cite{EL} and \cite{KO}).

 Successively, thanks to the left invariance of the standard Haar measure $\mu$ of the Heisenberg group $\mathbb{H}^n$, with respect to the natural action of the group $*:\mathbb{U}(n)\times HW^{1,2}_0(\Omega_\psi)\rightarrow HW^{1,2}_0(\Omega_\psi)$, given by
 $$\widehat{\tau}*q=(\tau z,t)\quad\mbox{for all $\widehat{\tau}=(\tau,1)\in \mathbb{U}(n)$,}\,\,q=(z,t)\in\mathbb H^n,$$
\noindent (see Chapter~III $\S$~2 No~4
of Bourbaki \cite{Btop}  and Chapter~7 $\S$~1 No~1 of Bourbaki \cite{BLie} for related topics) the principle of symmetric criticality of Palais, see Theorem~\ref{lem3} below, can be applied to the associated energy Euler--Lagrange functional
$$
I_{\lambda,\mu}(u,v)=\frac{1}{2}\left(\|u\|^2+\|v\|^2\right)-\lambda\int_{\Omega_\psi}\!\!K(q)F(u(q),v(q))d\mu(q)
$$
$$
\qquad-\mu\int_{\Omega_\psi} G(q,u(q),v(q)) d\mu(q),
$$
for every $(u,v)\in HW^{1,2}_{0,T}(\Omega_\psi)\times HW^{1,2}_{0,T}(\Omega_\psi)$,
 allowing a variational approach to
the problem~\eqref{p2}.

The methods used here may be suitable for other purposes, too. Indeed,
 we recall that a similar variational approach has been
used in a different context, in order to prove multiplicity results for elliptic problems defined in Euclidean strip--like domains and
involving the $p$--Laplacian operator (see \cite[Theorem 2.2]{K}). More precisely, in \cite{K} the author studied gradient--type systems of the form
\begin{equation*}\left\{\begin{aligned}
&-\Delta_{p}u= \lambda F_u(x,u,v)\quad{\rm in}\,\,\Omega\\
&-\Delta_{q}v= \lambda F_v(x,u,v)\, \quad{\rm in}\,\,\Omega\\
& \,\,\,u=v=0 \,\,{\rm on }\,\, \partial\Omega,
\end{aligned}\right.\end{equation*}
\noindent where
the nonlinearities $F_u$ and $F_v$ denote the partial derivatives of $F$ with respect to the second variable and the third variable, respectively, and
$\Omega=\omega\times\R^l$,  where $\omega$ is a bounded open subset of the Euclidean space $\R^m$ with smooth boundary, $m\geq 1$, $l\geq 2$ and $1<p,q<m+l$. Recently, nonlocal gradient--type systems have been investigated by exploiting similar variational arguments (see \cite[Theorem 3.1]{CaVi}). In both papers \cite{CaVi,K} a crucial role is played by some invariant subgroups of the orthogonal group $O(n)$ and Lions' embedding results (see \cite[Th\'{e}or\`{e}mes III.2 and III.3]{Lions}).

Theorem \ref{th1} extends the existence results obtained
in \cite{CaVi, K} to the Heisenberg group setting. In addition, in our case, the
presence of the Sobolev--Hardy term makes the search of weak solutions
much more delicate. Indeed, in order to handle the singular term in \eqref{p2}, it is crucial to introduce the subelliptic Hardy--Sobolev inequality
\begin{equation}\label{H}
\int_{\mathbb H^n}\psi^2 \dfrac{|\varphi|^2}{r^2}  d\mu(q) \le \left(\dfrac{2}{Q-2}\right)^2\int_{\mathbb H^n}|D_{\mathbb H^n}\varphi|_{\mathbb H^n}^2 d\mu(q),
\end{equation}
for every $\varphi \in C_0^\infty(\mathbb H^n\setminus\{O\})$, where the
main geometrical function $\psi$ is defined by
\begin{equation}\label{H2}
\psi (q)= |D_{\mathbb H^n}r|_{\mathbb H^n}=\frac {|z|} {r(q)} \quad \text{for all }  q=(z,t)\in \mathbb H^n,\text{ with }q\ne O,
\end{equation}
and $0\leq\psi  \leq 1 $, $\psi(0,t)\equiv 0$, $\psi(z,0)\equiv 1$.\par
Here, $r$ denotes the Heisenberg norm $r(q)=r(z,t)=(|z|^4+t^2)^{1/4} $, $z=(x,y) \in \mathbb C^{n}$, $t \in \mathbb R$.
 Furthermore, direct calculations show that
$$
\Delta_{\mathbb H^n}r=\frac {2n+1} {r} \psi^2 \quad \text{in } \ \mathbb H^n\setminus \{O\}.
$$
For details we refer to \cite[Section 2.1]{rigJDE}.

An important incentive to the study of subelliptic systems on the entire Heisenberg group $\mathbb{H}^n$ was recently provided by Pucci and co--authors  (\cite{BFP, BP, Pucci}), see also the papers \cite{fo, Gao, MaMi, ANONA1,a,b}.

In particular, the existence of nontrivial solutions for a subelliptic Schr\"{o}dinger--Hardy system
in the Heisenberg  group $\mathbb{H}^n$ was investigated in \cite{Pucci}, where the author considered the following problem
\begin{equation}\label{pppp1}\left\{\begin{aligned}
&-\Delta ^p_{\mathbb H^n}u+V(q)|u|^{p-2}u-\gamma\psi^p\frac{|u|^{p-2}u}{r(q)^{p}}=\lambda H_u(q,u,v)+\dfrac{\alpha}{p^*}|v|^{\beta}|u|^{\alpha-2}u\\
& -\Delta ^p_{\mathbb H^n} v+V(q)|v|^{p-2}v-\gamma\psi^p\frac{|v|^{p-2}v}{r(q)^p}=\lambda H_v(q,u,v)
+\dfrac{\beta}{p^*}|u|^{\alpha}|v|^{\beta-2}v,
\end{aligned}\right.\end{equation}
where $\gamma$ and $\lambda$ are real parameters, $Q=2n+2$ is the homogeneous dimension of the Heisenberg group $\mathbb H^n$, $1<p<Q$, the exponent $\alpha>1$ and $\beta>1$ are such that $\alpha+\beta=p^*$, $p^*=pQ/(Q-p)$,
and $\Delta^p_{\mathbb H^n}$ is the $p$--Laplacian operator on  $\mathbb H^n$,
which is defined by
\begin{equation*}
\Delta ^p_{\mathbb H^n}\varphi={\rm div}_H(|D_{\mathbb H^n}\varphi|^{p-2}_{\mathbb H^n}D_{\mathbb H^n}\varphi)
\end{equation*}
along any $\varphi\in C^\infty_0(\mathbb H^n)$, that is $\Delta^p_{\mathbb H^n}$
is the familiar horizontal $p$--Laplacian operator.
The potential function $V$ satisfy the following condition:

\noindent
$(\mathcal V)$\quad $V\in C(\mathbb H^n)$ {\it and} $\inf_{q\in\mathbb H^n}V(q)= V_0>0$.

The nonlinearities $H_u$ and $H_v$ denote the partial derivatives of $H$ with respect to the second variable and the third variable, respectively, and $H$ satisfies
\begin{itemize}
\item[$(\mathcal H)$] $H\in C^1(\mathbb{H}^n\times\mathbb{R}^2,\mathbb{R}^+)$, $H_z(q,0,0)=0$ {\it for all $q\in\mathbb{H}^n$ and there exist $\mu$ and $s$ such that $p<\mu\le s<p^*$
    and for every $\varepsilon>0$ there exists $C_\varepsilon>0$ for which the inequality
    $$|H_z(q,w)|\leq \mu\varepsilon|w|^{\mu-1}+qC_\varepsilon|w|^{s-1},\quad w=(u,v),\,\,|w|=\sqrt{u^2+v^2},$$
{\it holds for any } $(q,w)\in \mathbb{H}^n\times\mathbb{R}^2$, where $H_w=(H_u,H_v)$, and also}
\begin{align*}
0\leq\mu H(q,w)\leq H_w(q,w)\cdot w\ \ \mbox{\it for all } (q,w)\in\mathbb{H}^n\times \mathbb{R}^2
\end{align*}
{\it is valid.}
\end{itemize}

In this very interesting paper \cite{Pucci}, an existence result is obtained by an application of the mountain pass theorem and the Ekeland variational principle.
We emphasize that the assumptions adopted here are milder and in any case much different from
the usual hypotheses in problem \eqref{pppp1}. Moreover, to the best of our knowledge, Theorem \ref{th1} is
the first multiplicity result for subelliptic gradient--type systems on unbounded domains of $\mathbb{H}^n$.
In contrast with \cite{Pucci}, the Hilbertian setting, i.e. $p=2$, is peculiar for our approach in order to recover
the compactness properties stated in Lemma \ref{lem2} (see also Remark \ref{osservo}). In a forthcoming work
we plan to come back to problem \eqref{pppp1} and prove some multiplicity results by exploiting suitable group--theoretical
arguments and variational methods.

 Now, let us recall that the Folland--Stein horizontal Sobolev space $HW^{1,2}_0(\Omega_\psi)$ is the completion of $C^{\infty}_0(\Omega_\psi)$ with respect to the Hilbertian norm
   \begin{equation*}\label{Norma}\begin{gathered}
\|u\|_{HW^{1,2}_0(\Omega_\psi)}= \left(  \int_{\Omega_\psi} |D_{\mathbb H^n}u(q)|_{\mathbb H^n}^2 d\mu(q)+ \int_{\Omega_\psi} |u(q)|^2 d\mu(q)\right)^{1/2},\\
\langle u,\varphi\rangle= \int_{\Omega_\psi}\!\big( D_{\mathbb H^n}u(q), D_{\mathbb H^n}\varphi(q)\big)_{\mathbb H^n}\!d\mu(q)+ \int_{\Omega_\psi}u(q)\varphi(q) d\mu(q).
\end{gathered}
\end{equation*}

As far as we know, the abstract framework and Theorem \ref{th1} in the subelliptic setting are new also in the non--singular case.
A special and meaningful case of our main result reads as follows.

\begin{theorem} \label{th1}

Let $\Omega_\psi\subset \mathbb H^n$ as in \eqref{omega}, and let $K:\Omega_\psi\rightarrow \R$ be a potential satisfying $(h_K)$.
Furthermore, let $F:\R^2\rightarrow \R$ be a continuous function satisfying $(f_1)$--$(f_4)$.

Then there exist a number $\sigma>0$ and a nonempty open set $\Lambda\subset (0,\infty)$ such that, for every $\lambda\in \Lambda$, the following system
\begin{equation}\label{p23}\left\{\begin{aligned}
&-\Delta_{\mathbb H^n}u+u= \lambda K(q)\partial_1F(u,v)\quad{\rm in}\,\,\Omega_\psi\\
&-\Delta_{\mathbb H^n}v+v= \lambda K(q)\partial_2F(u,v)\, \quad{\rm in}\,\,\Omega_\psi\\
& \,\,\,u=v=0 \,\,{\rm on }\,\, \partial\Omega_\psi,
\end{aligned}\right.\end{equation}

\noindent has at least two solutions $(u^{(j)}_{\lambda,\mu},v^{(j)}_{\lambda,\mu})\in HW^{1,2}_{0,{\rm cyl}}(\Omega_\psi)\times HW^{1,2}_{0,{\rm cyl}}(\Omega_\psi)$, with $j\in \{1,2\}$, lying in the ball
$$
\{(u,v)\in HW^{1,2}_{0,{\rm cyl}
}(\Omega_\psi)\times HW^{1,2}_{0,{\rm cyl}
}(\Omega_\psi):\|(u,v)\|\leq \sigma\},
$$
where
$$
\|(u,v)\|=\|u\|_{HW^{1,2}_0(\Omega_\psi)}+\|v\|_{HW^{1,2}_0(\Omega_\psi)},
$$
and
\begin{equation*}
HW^{1,2}_{0,{\rm cyl}
}(\Omega_\psi)=\{u\in HW^{1,2}_0(\Omega_\psi): u(z,t)=u(|z|,t)\mbox{ for all  }q=(z,t)\in\Omega_\psi\},
\end{equation*}
is the linear subspace of cylindrically symmetric functions of $HW^{1,2}_0(\Omega_\psi)$.
\end{theorem}

\indent The manuscript is organized as follows. In Section~\ref{sec2}
we give some notations and we recall some properties of the functional space we work in.
In order to apply critical point methods to problem \eqref{p2}, we need to work in a subspace of the functional space $X=HW^{1,2}_0(\Omega_\psi)\times HW^{1,2}_0(\Omega_\psi)$.
 In particular, we give some tools which will be useful along the paper (see Lemmas \ref{crescitaF} and \ref{regularity}). Finally, in Section~\ref{MT} we study system~\eqref{p2} and prove our multiplicity result (see Theorem \ref{th1}). An example of an application is given in Example \ref{esempio}.

  \indent For general references on the subject and on methods used in the paper we refer
to the monographs \cite{BLU,KRV}, as  well as papers \cite{dAM,MZZ,MBF,V} and the references therein. See also \cite{GV,MBF,MolicaPucci,MBP,MBR} for related topics.

\section{Abstract framework}\label{sec2}
In this section we briefly recall some basic facts on the Heisenberg group and the functional Folland--Stein space~$HW^{1,2}_0(\Omega_\psi)$; see \cite{fol, folStein}.
The simplest example of Carnot
group of step two is provided by the Heisenberg group
$\mathbb H^n$ of topological dimension $m=2n+1$ and homogeneous dimension $Q=2n+2$, that is the Lie group whose underlying manifold is $\mathbb R^{2n+1}$, endowed with the non--Abelian group law
$$
q\circ q'=\biggl(z+z',t+t'+2\sum_{i=1}^n (y_ix_i'-x_iy_i')\biggr)
$$
for all $q$, $q'\in \mathbb H^n$, with
$$q=\!(z,t)\!=(x_1,\dots,x_n,y_1,\dots,y_n,t),\quad q'=\!(z',t')\!=(x_1',\dots,x_n',y_1',\dots,y_n',t').$$
The vector fields for $j=1,\dots,n$
\begin{equation}\label{eq:1}
X_j=\frac {\partial} {\partial x_j}+2y_j\frac {\partial} {\partial t}, \qquad Y_j=\frac {\partial} {\partial y_j}-2x_j\frac {\partial} {\partial t}, \qquad Z=\frac {\partial} {\partial t},
\end{equation}
constitute a basis $\mathcal B^*$ for the real graded Lie algebra $\mathfrak{H}=\bigoplus_{k=1}^{2}\mathfrak{H}_k$ of left invariant vector fields on $\mathbb H^n$. More precisely, the first graded component
$
\mathfrak{H}_1
$
is generated by $\mathcal B^*_1=\{X_j,Y_j:j=1,...,2n\}$
and the second graded component
$
\mathfrak{H}_2
$
is generated by $\mathcal B^*_2=\{Z\}$.
The basis $\mathcal B^*$ satisfies the Heisenberg canonical commutation relations for position and momentum
$[X_j,Y_k]=-4 \delta_{jk}{\partial}/ {\partial t}$, all other commutators are zero.

The natural inner product in the span of $\{X_j,Y_j\}_{j=1}^n$
$$
\big(W, Z\big)_{\mathbb H^n}=\sum_{j=1}^n \left( w^jz^j + \widetilde w^j\widetilde z^j\right)
$$
for  $W=\{ w^jX_j + \widetilde w^jY_j\}_{j=1}^n$ and $Z=\{z^jX_j + \widetilde z^jY_j\}_{j=1}^n$ produces the Hilbertian norm
$$
|D_{\mathbb H^n}u|_{\mathbb H^n}=\sqrt{\big(D_{\mathbb H^n}u,
D_{\mathbb H^n}u\big)_{\mathbb H^n}}
$$
for the horizontal vector field $D_{\mathbb H^n}u$. Moreover, if also $v\in C^1(\mathbb H^n)$ then  the  Cauchy--Schwarz inequality
$$
\big|\big(D_{\mathbb H^n}u,
D_{\mathbb H^n}v\big)_{\mathbb H^n}\big|_{\mathbb H^n}\leq |D_{\mathbb H^n}u|_{\mathbb H^n}|D_{\mathbb H^n}v |_{\mathbb H^n}
$$
continues to be valid.

For any horizontal vector field
$W=\{ w^jX_j + \widetilde w^jY_j\}_{j=1}^n$ of class $C^1(\mathbb H^n;\mathbb R^{2n})$ {\em the horizontal divergence} is defined by
$$
 {\rm div}_HW=\sum_{j=1}^n [X_j(w^j)+Y_j( \widetilde w^j)].
$$
If furthermore $g\in C^1(\mathbb R)$, then the {\em Leibnitz formula} holds, namely
$$
{\rm div}_H(gW)=g{\rm div}_H(W)+\big(D_{\mathbb H^n}g, W\big)_{\mathbb H^n}.
$$

If $u\in C^2(\mathbb H^n)$, then  the
{\em horizontal Laplacian in} $\mathbb H^n$ of $u$, called
the {\em Kohn--Spencer Laplacian},  is defined as  follows
\begin{align*}
\Delta_{\mathbb H^n}u&=\sum_{j=1}^n (X_j^2+Y_j^2)u\\
&=\sum_{j=1}^n \left(\frac{\partial^2} {\partial x_j^2}+\frac{\partial^2} {\partial y_j^2}+4y_j\frac{\partial^2} {\partial x_j\partial t}-4x_j\frac{\partial^2} {\partial y_j\partial t}\right)u+4|z|^2\frac{\partial^2 u} {\partial t^2},
 \end{align*}
and $\Delta_{\mathbb H^n}$ is {\em hypoelliptic} according to the celebrated  Theorem~1.1 due to {\em H\"ormander} ~\cite{hor}.

Going back to \eqref{p2}, we need to introduce a suitable solution space.
  Let $\Omega$ be a nontrivial open subset of $\mathbb H^n$. The Folland--Stein horizontal Sobolev space $HW^{1,2}_0(\Omega)$ is the completion of $C^{\infty}_0(\Omega)$, with respect to the Hilbertian norm
   \begin{equation}\label{Norma}\begin{gathered}
\|u\|_{HW^{1,2}_0(\Omega)}= \left(  \int_{\Omega} |D_{\mathbb H^n}u(q)|_{\mathbb H^n}^2 d\mu(q)+ \int_{\Omega} |u(q)|^2 d\mu(q)\right)^{1/2},\\
\langle u,\varphi\rangle= \int_{\Omega}\!\langle D_{\mathbb H^n}u(q), D_{\mathbb H^n} \varphi(q)\rangle_{\mathbb H^n}\!\,d\mu(q)+ \int_{\Omega}u(q)\varphi(q) d\mu(q).
\end{gathered}
\end{equation}

Of course, if $\Omega=\mathbb H^n$, then $HW^{1,2}(\mathbb H^n)=HW^{1,2}_0(\mathbb H^n)$, where $HW^{1,2}(\mathbb H^n) $ denotes the horizontal Sobolev space of the functions $u \in L^2(\mathbb H^n)$ such that $D_{\mathbb H^n}u$ exists in the sense of distributions and $|D_{\mathbb H^n}u|_{\mathbb H^n}$ is in $L^2(\mathbb H^n)$, endowed with the Hilbertian norm~\eqref{Norma}.

Since $\nu\in [0,C^{-1}_Vn^2)$, condition $(h_V)$ in addition to the Hardy--Sobolev inequality \eqref{H} and relation \eqref{H2}, gives that the norm $\|\cdot\|_{HW^{1,2}_0(\Omega_\psi)}$ which is equivalent to the norm given by
\begin{equation}\label{Normaeq}
\begin{gathered}
\|u\|=\left(\|u\|_{HW^{1,2}_0(\Omega_\psi)}^2-\nu\int_{\Omega_\psi}V(q)|u(q)|^2d\mu(q)\right)^{1/2},
\end{gathered}
\end{equation}
for every $u\in HW^{1,2}_0(\Omega_\psi)$.

More precisely, one has

\begin{equation*}\label{Normaeqrel}
\begin{gathered}
\frac{\sqrt{n^2-\nu C_V}}{n}\|u\|_{HW^{1,2}_0(\Omega_\psi)}\leq \|u\|\leq \|u\|_{HW^{1,2}_0(\Omega_\psi)},
\end{gathered}
\end{equation*}

\noindent for every $u\in HW^{1,2}_0(\Omega_\psi)$.

Take $q_1$, $q_2\in \mathbb H^n$ and let $H\Gamma_{q_1,q_2}(\mathbb H^n)$ be the set of piecewise smooth curves $\gamma$, such that $\gamma: [0,1]\rightarrow \mathbb H^n$, $\dot\gamma(t)\in \mathfrak{H}_1$ a.e. $t\in [0,1]$, $(\gamma(0),\gamma(1))=(q_1,q_2)$ and
 $$\int_0^{1}|\dot\gamma(t)|_{\mathbb H^n}dt<\infty.$$
Since $H\Gamma_{q_1,q_2}(\mathbb H^n)\neq \emptyset$ by the celebrated Chow--Rashevski\u\i\ theorem ~\cite{C}, it is possible to define the \textit{Carnot--Carath\'{e}odory distance} on $\mathbb H^n$, as follows
$$
d_{CC}(q_1,q_2)=\inf_{\gamma\in H\Gamma_{q_1,q_2}(\mathbb H^n)}\int_0^{1}|\dot\gamma(t)|_{\mathbb H^n}dt,
$$
see \cite{R} for details.\par

In order to use a variational approach for studying problem \eqref{p2}, we need to work in a
special functional space. Indeed, one of the difficulties in treating our problem is
related to the lack of a compact embedding of $HW^{1,2}_0(\Omega_\psi)$ into suitable Lebesgue spaces.
In this respect the standard subelliptic Sobolev spaces are not enough in
order to study the problem. We overcome this difficulty by working in a new functional
space, whose definition will be given below.

Let $(T, \cdot)$ be a closed topological group with neutral element
$\jmath$. The group $T$ is said to \textit{act continuously} on $\mathbb H^n$, if there exists a map $\star:T\times\mathbb H^n\rightarrow \mathbb H^n$ such that the following conditions hold:
\begin{itemize}
\item[$(T_1)$]$j\star q=q$ {\em for every} $q\in \mathbb H^n$;
\item[$(T_2)$]${\tau}_1\star({\tau}_2\star q)=({\tau}_1\cdot{\tau}_2)\star q$ {\em for every ${\tau}_1,$ ${\tau}_2\in T$ and} $q\in \mathbb H^n$.
\end{itemize}
In addition, the action $\star$ is called \textit{left distributed} if
\begin{itemize}
\item[$(T_3)$]${\tau}\star(p\circ q)=({\tau}\star p)\circ({\tau} \star q)$ {\em for every
${\tau}\in T$ and} $p,q\in \mathbb H^n$.
\end{itemize}
A set $\Omega\subset \mathbb H^n$ is said to be $T$--\textit{invariant}, with respect to the action $\star$, if
$T\star\Omega=\Omega$.

Let us consider $T=(T, \cdot)$ be a
closed infinite topological group acting continuously and left--distributively on $\mathbb H^n$ by the map $\star:T\times\mathbb H^n\rightarrow\mathbb H^n$. Assume
that $T$ acts isometrically on the horizontal Folland--Stein space $HW^{1,2}_0(\mathbb H^n)$, where the action $\sharp:T\times HW^{1,2}_0(\mathbb H^n)\rightarrow HW^{1,2}_0(\mathbb H^n)$ is defined
for every $({\widehat{\tau}},u)\in T\times HW^{1,2}_0(\mathbb H^n)$ by
\begin{equation*}\label{action}
({\widehat{\tau}}\sharp u)(q)=u({\widehat{\tau}}^{-1}\star q)\quad\mbox{for  all } q\in\mathbb H^n.
\end{equation*}
In what follows, $\mu$ is the natural Haar measure on $\mathbb H^n$ while $``\liminf"$ is the
Kuratowski lower limit of sets.

Let $\Omega$ be {a nonempty open $T$--invariant subset of $\mathbb H^n$, with boundary $\partial\Omega$, and assume that}
\begin{itemize}
\item[$(\mathcal{H})$] {\it for every $(q_k)_k\subset \mathbb H^n$ such that
 \begin{gather*}\lim_{k\rightarrow \infty}d_{CC}(e,q_k)=\infty\quad\mbox{and}\quad\mu\left(\displaystyle\liminf_{k\rightarrow \infty}(q_k\circ \Omega)\right)>0,\end{gather*}
 where $q_k\circ \Omega=\{q_k\circ q\,:\,q\in \Omega\}$,
    there exist a subsequence $(q_{k_{j}})_j$ of $(q_k)_k$ and a sequence of subgroups $(T_{q_{k_{j}}})_{j}$ of $T$, with cardinality $\mbox{\rm card}(T_{q_{k_{j}}})=\infty$, having the property that for all ${\widehat{\tau}}_1$, ${\widehat{\tau}}_2\in T_{q_{k_{j}}}$, with ${\widehat{\tau}}_1\neq{\widehat{\tau}}_2$, the following holds:
    $$
    \lim_{j\rightarrow \infty}\inf_{q\in \mathbb H^n}d_{CC}(({\widehat{\tau}}_1\star q_{k_j})\circ q,({\widehat{\tau}}_2\star q_{k_j})\circ q)=\infty.
    $$}
\end{itemize}

A domain $\Omega$ of $\mathbb H^n$, for which condition $(\mathcal H)$ holds, is simply called an
$\mathcal H$ {\it domain}. Let us denote $\mathbb{U}(n)=U(n)\times \{1\}$, where
$$
U(n)=U(n,\mathbb{C})=\{\tau\in GL(n;\mathbb{C})\,:\,\langle\tau z,\tau z'\rangle_{\mathbb C^n}= \langle z,z'\rangle_{\mathbb C^n} \mbox{ for all } z,z'\in \mathbb C^n\},
$$
that is, $U(n)$ is the usual unitary group. Here $\langle\cdot,\cdot\rangle_{\mathbb C^n}$ denotes the standard Hermitian product on $\mathbb C^n$, in other words $\langle z,z'\rangle_{\mathbb C^n}=\sum_{k=1}^nz_k\cdot\overline{z'_k}$.

Hence,  $\mathbb{U}(n)$ is the unitary group endowed with the natural multiplication law $\cdot:\mathbb{U}(n)\times \mathbb{U}(n)\rightarrow \mathbb{U}(n)$,
which acts continuously and left--distributively on $\mathbb H^n$  by the
map $* : \mathbb{U}(n)\times \mathbb{H}^n \to \mathbb{H}^n$, defined by
$$\widehat{\tau}*q=(\tau z,t)\quad\mbox{for all $\widehat{\tau}=(\tau,1)\in \mathbb{U}(n)$,\,\,}q=(z,t)\in\mathbb H^n,$$
thanks to~\cite[Lemma 3.1]{bk}. If we take $T=\mathbb{U}(n)$, then
$\Omega_\psi$ is $\mathbb{U}(n)$--invariant and an $\mathcal H$ domain, as shown in the proof of
of~\cite[theorem 1.1]{bk}. Moreover,
\begin{equation*}
HW^{1,2}_{0,\mathbb{U}(n)}(\Omega_\psi)=\{u\in HW^{1,2}_0(\Omega_\psi): u(z,t)=u(|z|,t)\mbox{ for all  }q=(z,t)\in\Omega_\psi\},
\end{equation*}
that is $HW^{1,2}_{0,\mathbb{U}(n)}(\Omega_\psi)=HW^{1,2}_{0, \rm cyl}(\Omega_\psi)$ is the space of cylindrically symmetric functions of $HW^{1,2}_0(\Omega_\psi)$.

Finally,
$\mathbb{U}(n)$ acts isometrically on the horizontal Folland--Stein space $HW^{1,2}_0(\mathbb H^n)$, where the action $\sharp:\mathbb{U}(n)\times HW^{1,2}_0(\mathbb H^n)\rightarrow HW^{1,2}_0(\mathbb H^n)$ is defined
for every $(\widehat{\tau},u)$ in~$\mathbb{U}(n)\times HW^{1,2}_0(\mathbb H^n)$ by
\begin{equation}\label{action}
(\widehat{\tau}\sharp u)(q)=u(\widehat{\tau}^{-1}*q)=u({\tau}^{-1}z,t)\quad\mbox{for  all } q=(z,t)\in\mathbb H^n,
\end{equation}
by force of~\cite[Lemma~3.2]{bk}.

Now, let $T=U(n_1)\times...\times U(n_\ell)\times\{1\}$, where $n=\sum_{i=1}^{\ell}n_i$, with $n_i\geq 1$ and $\ell\geq 1$ and consider the closed subspace
\begin{equation*}
HW^{1,2}_{0,T}(\Omega_\psi)=\{u\in HW^{1,2}_0(\Omega_\psi)\,:\, {\widehat{\tau}}\sharp u=u\mbox{ for all }
\widehat{\tau}\in T\}.
\end{equation*}
By keeping the same notation, we naturally
extend the function $u\in HW^{1,2}_{0,T}(\Omega_\psi)$ to the entire group $\mathbb{H}^n$ by zero on $\mathbb{H}^n\setminus \Omega_\psi$.

\subsection{Sobolev embedding results}

Following Folland and Stein \cite{folStein}, we can easily deduce the following embedding property that will be crucial in this paper.

\begin{lemma}\label{lem1}
Let $q \in [2,2^*]$ and $T=U(n_1)\times...\times U(n_\ell)\times\{1\}$, where $n=\sum_{i=1}^{\ell}n_i$, with $n_i\geq 1$ and $\ell\geq 1$. Then the embeddings
$$
HW^{1,2}_{0,T}(\Omega_\psi) \hookrightarrow HW^{1,2}_0(\Omega_\psi) \hookrightarrow L^q (\Omega_\psi)
$$
are continuous. Hence there exists a constant $k_q$ such that
\begin{equation*}
\|u\|_q=\|u\|_{L^q (\Omega_\psi)} \le k_q \|u\|_{HW^{1,2}_0(\Omega_\psi)}\quad\text{for all }
u \in HW^{1,2}_{0}(\Omega_\psi),
\end{equation*}
where $k_q$ depends on $q$ and $n$.

Moreover, since the norms $\|\cdot\|_{HW^{1,2}_0(\Omega_\psi)}$ and $\|\cdot\|$ are equivalent, there exists a constant $$c_q=k_q\frac{n}{\sqrt{n^2-\nu C_V}},$$ such that
\begin{equation}\label{eq3}
\|u\|_q\le c_q \|u\|,
\end{equation}
for every $u\in HW^{1,2}_{0}(\Omega_\psi)$.
\end{lemma}

On the other hand, by \cite{GarN,IV, Vas} we know that if $\mathcal O$
is a bounded open  set of $\mathbb H^n$ then the embedding
\begin{equation}\label{eq:25}
HW^{1,2}_0(\mathcal O)\hookrightarrow L^{q}(\mathcal O)
\end{equation}
is compact for all $q$, with $1\le q<2^*$.\par

Moreover, by \cite[Theorems 1.1 and 3.1]{bk}, the main compactness statement reads as follows.

\begin{lemma}\label{lem2}
Let $T=U(n_1)\times...\times U(n_\ell)\times\{1\}$, where $n=\sum_{i=1}^{\ell}n_i$, with $n_i\geq 1$ and $\ell\geq 1$. Then the embedding $$HW^{1,2}_{0,T}(\Omega_\psi)\hookrightarrow L^q(\Omega_\psi)$$ is compact for any $q\in (2,2^*)$. We also have that
$$
HW^{1,2}_{0,T}(\Omega_\psi)=Fix_T(HW^{1,2}_0(\Omega_\psi)),
$$
where
$$
Fix_T(HW^{1,2}_0(\Omega_\psi))=\{u\in HW^{1,2}_0(\Omega_\psi): {\widehat{\tau}}\sharp u=u\,\, \mbox{\rm for all } \widehat{\tau}\in T\},
$$
and
$\sharp:T\times HW^{1,2}_0(\mathbb H^n)\rightarrow HW^{1,2}_0(\mathbb H^n)$ is the action defined in \eqref{action}.
\end{lemma}

\begin{remark}\label{osservo}\rm{The above lemma is a consequence of a more general result stated in \cite[Theorem 3.1]{bk}.
More precisely,
let $\mathbb{G}=(\mathbb{G}, \circ)$ be a Carnot group of step $r$ and homogeneous dimension $Q>2$, with the neutral element denoted by $e$. Let $T=(T, \cdot)$ be a
closed infinite topological group acting continuously and left distributively on $\mathbb{G}$ by the map $\circledast:T\times \mathbb G\rightarrow\mathbb G$. Assume furthermore
that $T$ acts isometrically on $HW^{1,2}_0(\mathbb{G})$, where the action $\natural: T\times HW^{1,2}_0(\mathbb{G})\rightarrow HW^{1,2}_0(\mathbb{G})$ is~defined by
\begin{equation*}
(\widehat{\tau}\natural u)(q)=u(\widehat{\tau}^{-1}\circledast q)\quad\mbox{for  all } q\in\mathbb G.
\end{equation*}
Let $\mathbb{G}_0$ be an $\mathcal{H}$ domain (see \cite{MolicaPucci}). Then, the following embedding
$$
HW^{1,2}_{0,T}(\mathbb{G}_0)\hookrightarrow L^q (\mathbb{G}_0)
$$
is compact for every $q \in (2,2^*)$. This result was inspired by Tintarev and Fieseler \cite{TF}.}
\end{remark}

\subsection{Weak formulation and $T$--invariance}\label{sottosezione}

The natural solution space for \eqref{p2} is
$$\begin{aligned}
X=HW^{1,2}_0(\Omega_\psi)\times HW^{1,2}_0(\Omega_\psi),
\end{aligned}$$
with associated norm
$$\begin{aligned}
\|(u,v)\|=\|u\|+\|v\|.
\end{aligned}$$
Let us consider the action $\pi_{\sharp}:T\times X\rightarrow X$ given by
$$
\pi_{\sharp}(\widehat{\tau}, (u,v))=(\widehat{\tau}\sharp u, \widehat{\tau}\sharp v),
$$
for every $\widehat{\tau}\in T$ and $(u,v)\in X$.

The above definition immediately yields
$$
Fix_T(X)=Fix_T(HW^{1,2}_0(\Omega_\psi))\times Fix_T(HW^{1,2}_0(\Omega_\psi)),
$$
where
$$
Fix_T(X)=\{(u,v)\in X:\pi_{\sharp}(\widehat{\tau}, (u,v))=(u,v)\,\, \mbox{\rm for all } \widehat{\tau}\in T\}.
$$

A function $(u,v)\in X$  is said to be a (\textit{weak})
{\em  solution} of problem~\eqref{p2} if
\begin{align}\label{weaksol}
\langle u,\varphi\rangle_{X,\nu}+\langle v,\psi\rangle_{X,\nu} & = \lambda\int_{\Omega_\psi}K(q)\partial_1F(u(q),v(q))\varphi(q) d\mu(q)\nonumber\\
& \quad+\lambda\int_{\Omega_\psi}K(q)\partial_2F(u(q),v(q))\varphi(q) d\mu(q)\\
& \quad+\mu\int_{\Omega_\psi}\partial_1G(q,u(q),v(q))\psi(q) d\mu(q)\nonumber\\
& \quad+\mu\int_{\Omega_\psi}\partial_2G(q,u(q),v(q))\psi(q) d\mu(q)\nonumber,
\end{align}
\noindent for any $(\varphi,\psi)\in X$, where we set
$$
\langle u,\varphi\rangle_{X,\nu}=\langle u,\varphi\rangle-\nu \int_{\Omega_\psi}V(q)u(q)\varphi(q)d\mu(q),
$$
and
$$
\langle v,\psi\rangle_{X,\nu}=\langle v,\psi\rangle-\nu \int_{\Omega_\psi}V(q)v(q)\psi(q)d\mu(q).
$$

\indent Problem \eqref{p2} has a variational nature
and the Euler--Lagrange functional $I_{\lambda,\mu}:X\rightarrow \R$ associated to~\eqref{p2}~is given by
$$
I_{\lambda,\mu}(u,v)=\frac{1}{2}\left(\|u\|^2+\|v\|^2\right)-\lambda\int_{\Omega_\psi}\!\!K(q)F(u(q),v(q))d\mu(q)
$$
$$
\qquad-\mu\int_{\Omega_\psi} G(q,u(q),v(q)) d\mu(q),
$$
for every $(u,v)\in X$.\par
Clearly, the functional $I_{\lambda,\mu}$ is well--defined in $X$ and, thanks to $(f_1)$--$(f_3)$ as well as $(g_1)$--$(g_3)$, it is of class $C^1(X)$. Moreover, for every $(u,v)\in X$
\begin{align}\label{weakform}
\langle I_{\lambda,\mu}'(u,v),(\varphi,\psi)\rangle&= \langle u,\varphi\rangle-\nu \int_{\Omega_\psi}V(q)u(q)\varphi(q)d\mu(q)\nonumber\\
&\quad + \langle v,\psi\rangle-\nu \int_{\Omega_\psi}V(q)v(q)\psi(q)d\mu(q)\nonumber\\
& \quad- \lambda\int_{\Omega_\psi}K(q)\partial_1F(u(q),v(q))\varphi(q) d\mu(q)\\
& \quad-\lambda\int_{\Omega_\psi}K(q)\partial_2F(u(q),v(q))\varphi(q) d\mu(q)\nonumber\\
& \quad-\mu\int_{\Omega_\psi}\partial_1G(q,u(q),v(q))\psi(q) d\mu(q)\nonumber\\
& \quad-\mu\int_{\Omega_\psi}\partial_2G(q,u(q),v(q))\psi(q) d\mu(q)\nonumber,
\end{align}
for all $(\varphi,\psi)\in  X$.
Hence, the critical points of $I_{\lambda,\mu}$ in~$X$ are exactly the (weak) solutions of~\eqref{p2}.\par

Let
$$\begin{aligned}
Y_T=HW^{1,2}_{0,T}(\Omega_\psi)\times HW^{1,2}_{0,T}(\Omega_\psi)\subset X
\end{aligned}$$
be endowed with the induced norm $\|\cdot\|$, where $T=U(n_1)\times...\times U(n_\ell)\times\{1\}\subseteq \mathbb{U}(n)$, with $n=\sum_{i=1}^{\ell}n_i$, $n_i\geq 1$ and $\ell\geq 1$.

 A pair $(u,v)\in Y_T$ is said to be a (\textit{weak}) \textit{solution of problem~\eqref{p2} only in the $Y_T$ sense} if equality \eqref{weaksol} holds for
every $(\varphi,\psi)\in  Y_T$.\par
Then $(u,v)\in Y_T$ is a solution of \eqref{p2} in the entire space $HW^{1,2}_0(\Omega_\psi)$
if the {\em principle of symmetric criticality} of Palais given in~\cite{pala}
holds.

 To prove this let us recall the well known principle of symmetric criticality of Palais
stated in the general form, proved in~\cite{demo} for reflexive strictly convex Banach spaces. For details and comments we refer to ~\cite[Section~5]{cp}.

More precisely, let $E=(E,\|\cdot\|_E)$ be a reflexive strictly convex Banach space. Suppose that $\mathcal G$ is a subgroup of isometries $g: E\to E$, that is $g$ is linear and $\|{gu}\|_E=\|{u}\|_E$ for all $u\in E$.\par
 Consider the $\mathcal G$--invariant closed subspace of $E$
\begin{equation*}
\Sigma_{\mathcal G}=\{u\in E: gu=u\mbox{ for all }g\in\mathcal G\}.
\end{equation*}
By \cite[Proposition 3.1]{demo} we have the following result.
\begin{theorem}\label{lem3}
Let $E$, $\mathcal G$ and $\Sigma$ be as before and let $I$ be a $C^1$ functional defined on $E$ such that $$I(gu)=I(u),\,\,\,\,\forall\, u\,\in E$$ for every $g\in \mathcal G$.

 Then $u\in\Sigma_{\mathcal G}$ is a critical point of $I$ if and only if $u$ is a critical point of $\mathcal J=I\vert_{\Sigma_{\mathcal G}}$.
\end{theorem}
We recall that $\Omega_\psi$ is a nonempty open subset of $\mathbb H^n$, which is
$T$--invariant. Furthermore, we recall that from the invariance point of view, the unitary groups play
the same role in the Heisenberg setting as the orthogonal groups in the Euclidean framework.

Thus we apply the principle of symmetric criticality to the Sobolev space $Y_T$ under the action
$\pi_{\sharp}:T\times X\rightarrow X$ defined in \eqref{action}. Let us again denote
$$
\Phi(u,v)=\frac{1}{2}\left(\|u\|^2+\|v\|^2\right),\quad\forall\, (u,v)\in X.
$$
\indent Clearly,
\begin{equation}\label{invariance}
\begin{split}
\Phi(\pi_{\sharp}(\widehat{\tau}, (u,v)))& = \Phi(\widehat{\tau}\sharp u, \widehat{\tau}\sharp v)\\
&=\frac{1}{2}\big(\|\widehat{\tau}\sharp u\|^2+\|\widehat{\tau}\sharp v\|^2\big)\\
&=\frac{1}{2}\big(\|u\|^2+\|v\|^2\big)\\
&=\Phi(u,v),\quad\mbox{for every}\, (u,v)\in X,\,\,\widehat{\tau}\in T,
\end{split}
\end{equation}
since $T$ acts isometrically on $HW^{1,2}_0(\Omega_\psi)$ as proved in \cite[Lemma 3.2]{bk}. Thus the functional
$\Phi$
is $T$--invariant.\par
Moreover, the functional $\Upsilon_{\lambda,\mu}:X\rightarrow\mathbb{R}$ given by
\begin{equation*}
\begin{split}
\Upsilon_{\lambda,\mu}(u,v)& = \lambda\int_{\Omega_\psi}\!\!K(q)F(u(q),v(q))d\mu(q)\\
& \quad+\mu\int_{\Omega_\psi} G(q,u(q),v(q)) d\mu(q),\quad\mbox{for every}\,\, (u,v)\in X,\,\, \widehat{\tau}\in T,
\end{split}
\end{equation*}

\noindent is $T$--invariant
by assumptions $(h_K)$ and $(g_4)$.

Indeed, let us fix $\widehat{\tau}\in T$ and $(u,v)\in X$. Then
putting $\tau^{-1}*q=p$, we get by $(T_1)$--$(T_3)$
\begin{align*}
\Upsilon_{\lambda,\mu}(\pi_{\sharp}(\widehat{\tau}, (u,v)))&=\Upsilon_{\lambda,\mu}(\widehat{\tau}\sharp u, \widehat{\tau}\sharp v)\\
&=\lambda\int_{\Omega_\psi}\!\!K(q)F(\widehat{\tau}\sharp u(q),\widehat{\tau}\sharp v(q))d\mu(q)
+\mu\int_{\Omega_\psi} G(q,\widehat{\tau}\sharp u(q),\widehat{\tau}\sharp v(q)) d\mu(q)\\
&=\lambda\int_{\Omega_\psi}\!\!K(q)F(u({\widehat{\tau}}^{-1}*q),v({\widehat{\tau}}^{-1}*q))d\mu(q)\\
&\quad +\mu\int_{\Omega_\psi} G(q,u({\widehat{\tau}}^{-1}*q),v({\widehat{\tau}}^{-1}*q)) d\mu(q)\\
&=\lambda\int_{\widehat{\tau}*\Omega_\psi}\!\!K(\widehat{\tau}*p)F(u(p),v(p))d\mu(\widehat{\tau}*p)\\
&\quad +\mu\int_{\widehat{\tau}*\Omega_\psi} G(\widehat{\tau}*p,u(p),v(p)) d\mu(\widehat{\tau}*p)\\
&=\lambda\int_{\Omega_\psi}\!\!K(p)F(u(p),v(p))d\mu(p)+\mu\int_{\Omega_\psi} G(p,u(p),v(p)) d\mu(p)\\
&=\Upsilon_{\lambda,\mu}(u,v),
\end{align*}
since $T*\Omega_\psi=\Omega_\psi$ and $K,G$ are $T$--invariant by assumption. Moreover,
the left $*$  invariance of the measure $\mu$ (keeping in mind that the Jacobian
of the change of variables has determinant 1) implies
$$d\mu(\widehat{\tau}*p)=d\mu(p)\quad\mbox{for all }p\in\Omega_\psi,$$
which is exactly formula (10) from~\cite{BLie}, where 1 is the multiplier of~$\mu$.
See also~\cite[Chapter~4]{BBE}.

Thus,  $I_{\lambda,\mu}$ is $T$--invariant in $X$ with respect to the action
$\pi_{\sharp}:T\times X\rightarrow X$.

Hence, the principle of
symmetric criticality of Palais ensures that $(u,v)\in Y_T$ is a solution of problem \eqref{p2} if and only if $(u,v)$ is a critical point of the functional $\mathcal{J}_{\lambda,\mu}:Y_T\rightarrow \mathbb{R}$, where $\mathcal{J}_{\lambda,\mu}=I_{\lambda,\mu}\vert_{Y_T}$.

We will employ
first an abstract theorem by Ricceri \cite[Theorem 4]{Ricceri} merging together minimax
and critical point theory to derive the existence of two local minima for the energy
$\mathcal{J}_{\lambda,\mu}$. For the convenience of the reader and to make our exposition self--contained, we
state this abstract tool below, is the version rephrased in terms of the weak topology.

\begin{theorem}\label{teorricceri}
Let $E$ be a reflexive Banach space, $D\subseteq \R$ an interval, and $\Psi:E\times D\rightarrow \R$ a function satisfying the following:
\begin{itemize}
\item[$(\Psi_1)$] $\Psi(x,\cdot)$ is concave in $D$ for every $x\in E$$;$
\item[$(\Psi_2)$] $\Psi(\cdot,\lambda)$ is continuous, coercive and sequentially weakly lower semicontinuous in $E$, for every $\lambda\in D$$;$
\item[$(\Psi_3)$] $\displaystyle\sup_{\lambda\in D}\inf_{x\in E}\Psi(x,\lambda)<\displaystyle\inf_{x\in E}\sup_{\lambda\in D}\Psi(x,\lambda)$.
\end{itemize}
Then for every $\zeta>\displaystyle\sup_{\lambda\in D}\inf_{x\in E}\Psi(x,\lambda)$, there exists an a nonempty open set $\Lambda\subseteq D$ with the following property$:$

for every $\lambda \in \Lambda$ and every sequentially weakly lower semicontinuous functional $\Theta:E\rightarrow \R$, there exists $\delta>0$ such that, for every $\mu\in (0,\delta)$, the functional $\mathcal{E}_{\lambda,\mu}:E\rightarrow \R$ given by
$$
\mathcal{E}_{\lambda,\mu}(x)=\Psi(x, \lambda)+\mu \Theta(x)
$$
has at least two local minima lying in the set
$$
E^{\lambda}_\zeta=\{x\in E:\Psi(x,\lambda)<\zeta\}.
$$
\end{theorem}

\subsection{Growth condition and regularity properties}

Here, we use the structural assumptions on $F$ to get some bounds from above for the nonlinear term and
its derivatives. This part is quite standard and does not take into account the subelliptic features of the problem: the reader
familiar with these nonlinear analysis estimates may go directly to Lemma \ref{regularity}.

\begin{lemma}\label{crescitaF}
Assume that conditions $(f_1)$--$(f_3)$ hold. Then for every $\varepsilon>0$ there exists $c_\varepsilon>0$ such that
\begin{itemize}
\item[$(i)$] $
\max\{|\partial_1 F(\eta,\zeta)|, |\partial_2 F(\eta,\zeta)|\}\leq \varepsilon \left(|\eta|+ |\zeta|\right) + c_\varepsilon \left(|\eta|^{\alpha-1} + |\zeta|^{(\alpha-1)} \right),$
\item[$(ii)$] for every $(\eta,\zeta)\in\R^2$ \begin{equation*}
\begin{split}
|F(\eta,\zeta)|&\leq \varepsilon \left(|\eta|^{2} + 2|\eta||\zeta| + |\zeta|^2\right)\\
&\;\;\;  + c_\varepsilon \left(|\eta|^{\alpha} + |\zeta|^\alpha+|\zeta|^{(\alpha-1)}|\eta|  + |\zeta||\eta|^{(\alpha-1)} \right).
\end{split}
\end{equation*}
\end{itemize}
\end{lemma}

\begin{proof}
Let $\varepsilon>0$. First, we will prove that there exists $c_{\varepsilon,1}>0$ such that
\begin{equation}\label{ine1}
|\partial_1 F(\eta,\zeta)|\leq \varepsilon \left(|\eta|+ |\zeta|\right) + c_{\varepsilon,1} \left(|\eta|^{\alpha-1} + |\zeta|^{(\alpha-1)} \right),
\end{equation}
for every $(\eta,\zeta)\in\R^2$.
Indeed, by $(f_2)$, it follows in particular, that
$$
\lim_{\eta,\zeta\rightarrow 0}\frac{\partial_1 F(\eta,\zeta)}{|\eta|+|\zeta|}=0.
$$
Thus there exists $\delta_\varepsilon>0$ such that if $|\eta|+|\zeta|<\delta_\varepsilon$, then $|\partial_1 F(\eta,\zeta)|\leq \varepsilon (|\eta|+|\zeta|)$. On the other hand, if $|\eta|+|\zeta|\geq\delta_\varepsilon$, condition $(f_3)$ yields
\begin{equation}\label{diseq1}
\begin{split}
|\partial_1 F(\eta,\zeta)|&\leq \epsilon \left(|\eta|+|\zeta|+|\eta|^{\alpha-1}\right)\\
&= \epsilon \left((|\eta|+|\zeta|)^{\alpha-1}(|\eta|+|\zeta|)^{2-\alpha}+|\eta|^{\alpha-1}\right)\\
& \leq \epsilon \left((|\eta|+|\zeta|)^{\alpha-1}\delta_\varepsilon^{2-\alpha}+|\eta|^{\alpha-1}\right).
\end{split}
\end{equation}
Moreover, since $\alpha\in (2,2^*)$ and bearing in mind that
$$(|\eta|+|\zeta|)^{\alpha-1}\leq 2^{\alpha-2}(|\eta|^{\alpha-1}+|\zeta|^{\alpha-1}),\quad \forall\, (\eta,\zeta)\in \R^2$$ inequality \eqref{diseq1} gives
\begin{equation}\label{diseq2}
|\partial_1 F(\eta,\zeta)|\leq c_{\varepsilon,1} \left(|\eta|^{\alpha-1}+|\zeta|^{\alpha-1}\right),
\end{equation}
for every $(\eta,\zeta)\in\R^2$ with $|\eta|+|\zeta|\geq\delta_\varepsilon$. Hence, relation \eqref{ine1} immediately follows by \eqref{diseq1} and \eqref{diseq2}. Arguing as above, it follows that
there exists $c_{\varepsilon,2}>0$ such that
\begin{equation}\label{ine2}
|\partial_2 F(\eta,\zeta)|\leq \varepsilon \left(|\eta|+ |\zeta|\right) + c_{\varepsilon,2} \left(|\eta|^{\alpha-1} + |\zeta|^{(\alpha-1)} \right),
\end{equation}
for every $(\eta,\zeta)\in\R^2$.
In conclusion, relation $(i)$ holds by \eqref{ine1} and \eqref{ine2} if taking $c_\varepsilon=\max\{c_{\varepsilon,1}, c_{\varepsilon,2}\}$.

In order to prove part $(ii)$, we make use of the Mean Value Theorem in two variables. More precisely, by $(f_1)$ it follows that
\begin{equation}\label{Lagrange}
\begin{split}
|F(\eta,\zeta)|&= |F(\eta,\zeta)-F(0,0)|\\
&= |\nabla F(c\eta,c\zeta)\cdot (\eta,\zeta)|\\
& \leq |\partial_1 F(c\eta,c\zeta)||\eta|+|\partial_2 F(c\eta,c\zeta)||\zeta|,
\end{split}
\end{equation}
for every $(\eta,\zeta)\in \R^2$ and some $c\in (0,1)$. Now, by using part $(ii)$ it follows that
\begin{equation}\label{Lagrange2}
\max\{|\partial_1 F(c\eta,c\zeta)|, |\partial_2 F(c\eta,c\zeta)|\}\leq \varepsilon \left(|\eta|+ |\zeta|\right) + c_\varepsilon \left(|\eta|^{\alpha-1} + |\zeta|^{(\alpha-1)} \right),
\end{equation}
for every  $(\eta,\zeta)\in \R^2$. A direct computation shows that $(ii)$ follows by \eqref{Lagrange} and \eqref{Lagrange2}.
\end{proof}

The next result is a consequence of Lemma \ref{crescitaF} and can be viewed as a counterpart of some contributions obtained in several different contexts (see, among others, the paper \cite{K}) to the case of subelliptic gradient--type systems defined on the domain $\Omega_\psi$. We emphasize that a key ingredient of the proof is given by Lemmas \ref{lem1} and \ref{lem2}. They express peculiar and intrinsic aspects of the problem under consideration.

\begin{lemma}\label{regularity}
Let $T=U(n_1)\times...\times U(n_\ell)\times\{1\}$, where $n=\sum_{i=1}^{\ell}n_i$, with $n_i\geq 1$ and $\ell\geq 1$. Furthermore, let $\Omega_\psi\subset \mathbb H^n$ as in \eqref{omega}, $K:\Omega_\psi\rightarrow \R$ be a potential such that $(h_K)$ holds, and $F:\R^2\rightarrow \R$ be a continuous function satisfying $(f_1)$--$(f_3)$.

Then the functional $\mathfrak F:Y_T\rightarrow \R$ given by
$$
\mathfrak F(u,v)=\int_{\Omega_\psi}K(q)F(u(q),v(q))d\mu(q),\quad\forall\, (u,v)\in Y_T
$$
is sequentially weakly continuous on $Y_T$.
\end{lemma}

\begin{proof}
In order to prove that $\mathfrak F$ is a sequentially weakly continuous functional, arguing by contradiction, we assume that there exists a sequence
$\{(u_j,v_j)\}_{j\in \mathbb N}\subset Y_T$ which weakly converges to an element $(\widetilde{u},\widetilde{v})\in Y_T$, and such that
\begin{equation}\label{wslsc}
|\mathfrak F(u_j,v_j)-\mathfrak F(\widetilde{u},\widetilde{v})|>\varepsilon_0,
\end{equation}
for every $j\in \mathbb{N}$ and some $\varepsilon_0>0$. Now, fixing $(u,v)\in Y_T$, one has
\begin{equation}\label{derivata}
\begin{split}
\mathfrak{F}'(u,v)(\varphi,\psi) &=\int_{\Omega_\psi}K(q)\partial_1F(u(q),v(q))\varphi(q)d\mu(q)\\
&\quad +\int_{\Omega_\psi}K(q)\partial_2F(u(q),v(q))\phi(q)d\mu(q),\quad\forall\, (\varphi,\psi)\in Y_T.
\end{split}
\end{equation}

Invoking \eqref{wslsc}, the Mean Value Theorem ensures that
\begin{equation}\label{derivata}
0<\varepsilon_0\leq |\mathfrak{F}'(w_j,y_j)(u_j-\widetilde{u},v_j-\widetilde{v})|,
\end{equation}
where
$$
w_j=u_j+\theta_j(\widetilde{u}-u_j),
$$
and
$$
y_j=v_j+\theta_j(\widetilde{v}-v_j),
$$
 for some $\theta_j\in (0,1)$, for every $j\in \mathbb{N}$.\par
 By using Lemma \ref{crescitaF} (part $(i)$), the H\"{o}lder inequality yields
\begin{equation}\label{BoX}
\begin{split}
  |\mathfrak{F}'(w_j,y_j)(u_j-\widetilde{u},v_j-\widetilde{v})| & \leq\int_{\Omega_\psi}K(q)|\partial_1F(w_j(q),y_j(q))||u_j(q)-\widetilde{u}(q)|d\mu(q)\\
  &\quad +\int_{\Omega_\psi}K(q)|\partial_2F(w_j(q),y_j(q))||v_j(q)-\widetilde{v}(q)|d\mu(q)\\
  &\leq \varepsilon \Bigg(\int_{\Omega_\psi}K(q)(|w_j(q)|+|y_j(q)|)|u_j(q)-\widetilde{u}(q)|d\mu(q)\\
   &\quad +\int_{\Omega_\psi}K(q)(|w_j(q)|+|y_j(q)|)|v_j(q)-\widetilde{v}(q)|d\mu(q)\Bigg)\\
   &\quad +c_\varepsilon \Bigg(\int_{\Omega_\psi}K(q)(|w_j(q)|^{\alpha-1}+|y_j(q)|^{\alpha-1})|u_j(q)-\widetilde{u}(q)|d\mu(q)\\
   &\quad +\int_{\Omega_\psi}K(q)(|w_j(q)|^{\alpha-1}+|y_j(q)|^{\alpha-1})|v_j(q)-\widetilde{v}(q)|d\mu(q)\Bigg)\\
   &\leq \varepsilon\|K\|_\infty\Bigg((\|w_j\|_2+\|y_j\|_2)(\|u_j-\widetilde{u}\|_2+\|v_j-\widetilde{v}\|_2)\Bigg)\\
   &\quad +c_\varepsilon\|K\|_\infty\Bigg((\|w_j\|_\alpha^{\alpha-1}+\|y_j\|_\alpha^{\alpha-1})(\|u_j-\widetilde{u}\|_\alpha+\|v_j-\widetilde{v}\|_\alpha)\Bigg).
\end{split}
\end{equation}

Now, it is easy to note that the sequences $\{w_j\}_{j\in \N}$ and $\{y_j\}_{j\in \N}$ are bounded in $HW^{1,2}_{0,T}(\Omega_\psi)$. Moreover, due to the compactness Lemma \ref{lem2}, $u_j\rightarrow \widetilde{u}$ and $v_j\rightarrow \widetilde{v}$ in $L^{\alpha}(\Omega_\psi)$. Consequently, the
last expression in \eqref{BoX} tends to zero and this fact
contradicts (\ref{wslsc}).

 In conclusion, the functional $\mathfrak{F}$ is sequentially
weakly continuous and this completes the proof.
\end{proof}

\section{Main multiplicity result}\label{MT}

With the previous notations, the main result of the present paper reads as follows.

\begin{theorem} \label{th1}
Let $T=U(n_1)\times...\times U(n_\ell)\times\{1\}$, where $n=\sum_{i=1}^{\ell}n_i$, with $n_i\geq 1$ and $\ell\geq 1$.
Let $\Omega_\psi\subset \mathbb H^n$ be as in \eqref{omega} with $O=(0,0)\in \Omega_\psi$, $\nu\in [0, C_V^{-1}n^2)$ fixed, and let $V,K:\Omega_\psi\rightarrow \R$ be potentials satisfying $(h_V)$ and $(h_K)$.
Furthermore, let $F:\R^2\rightarrow \R$ be a continuous function satisfying $(f_1)$--$(f_4)$.

Then there exist a number $\sigma>0$ and a nonempty open set $\Lambda\subset (0,\infty)$ such that, for every $\lambda\in \Lambda$ and every continuous function $G:\Omega_\psi\times\R^2\rightarrow \R$ satisfying $(g_1)$--$(g_4)$, there exists $\mu_0>0$ such that, for each $\mu\in (0,\mu_0)$, the gradient--type system \eqref{p2} has at least two weak solutions $(u_{\lambda,\mu}^{(j)}, v_{\lambda,\mu}^{(j)})\in Y_T$, with $j\in \{1,2\}$, lying in the ball
$$
\{(u,v)\in Y_T:\|(u,v)\|\leq \sigma\},
$$
where
\begin{equation*}
\begin{split}
\|(u,v)\|&=\left(\|u\|_{HW^{1,2}_0(\Omega_\psi)}^2-\nu\int_{\Omega_\psi}V(q)|u(q)|^2d\mu(q)\right)^{1/2}\\
&\qquad\qquad +
\left(\|v\|_{HW^{1,2}_0(\Omega_\psi)}^2-\nu\int_{\Omega_\psi}V(q)|v(q)|^2d\mu(q)\right)^{1/2}.
\end{split}
\end{equation*}
\end{theorem}
\begin{proof}
We will show that the assumptions of Theorem \ref{teorricceri} are fulfilled by choosing $E=Y_T$ and $D=[0,\infty)$. Moreover, let us denote
$$
\Phi(u,v)=\frac{1}{2}\left(\|u\|^2+\|v\|^2\right),
$$
and
$$
\mathfrak F(u,v)=\int_{\Omega_\psi}K(q)F(u(q),v(q))d\mu(q),
$$
for every $(u,v)\in Y_T$, and define $\Psi:Y_T\times D\rightarrow \R$ as follows:
$$
\Psi((u,v),\lambda)=\Phi(u,v) - \lambda\mathfrak F(u,v) + \lambda\rho_0
$$
for every $(u,v)\in Y_T$, and $\lambda\in D$. Here $\rho_0$ is a positive and sufficiently small real parameter (see \eqref{condrho} below).

We observe that the function $\Psi((u,v),\cdot)$ is concave in $D$, for every $(u,v)\in Y_T$. Moreover, the functional
$\Psi(\cdot,\lambda)$ is continuous and sequential weak lower semicontinuity on $Y_T$, for every $\lambda\in D$.

Moreover, fixing $\lambda\in D$, on account of $(f_4)$, fixing $\mu_F,\nu_F\in [2,2^*]$, by H\"{o}lder's inequality and Lemma \ref{lem1}, one has
\begin{align*}
\Psi((u,v),\lambda) & \geq \Phi(u,v)-\lambda\left(\int_{\Omega_\psi} \left(\kappa_1K(q)|u(q)|^{p_F}+\kappa_2K(q)|v(q)|^{q_F} +\kappa_3K(q) \right) d\mu(q) \right)  \\
&\geq \Phi(u,v) -\lambda \left(\kappa_1\left\|K\right\|_{\mu_F/(\mu_F-p_F)}\left\| u\right\|_{\mu_F}^{p_F} + \kappa_2\left\|K\right\|_{\nu_F/(\nu_F-q_F)}\left\| v\right\|_{\nu_F}^{q_F} +\kappa_3\left\|K\right\|_1 \right) \\
& \geq\Phi(u,v) -c\lambda \left(\left\| u\right\|^{p_F} + \left\| v\right\|^{q_F} +1 \right)
\end{align*}
for some $c>0$. Thus
$$
\lim_{\|(u,v)\|\rightarrow \infty}\Psi((u,v),\lambda)=\infty,
$$
since $\max\{p_F,q_F\}<2$.

Hence $(\Psi_1)$ and $(\Psi_2)$ of Theorem \ref{teorricceri} are verified. Next, we deal with $(\Psi_3)$. First, let us consider the real function $f:(0,\infty)\to\R$ defined by
$$
f(\xi)=\sup_{(u,v)\in \Phi^{-1}((-\infty,\xi])}\mathfrak F(u,v),
$$
for every $\xi\in \R$.\par
By Lemma \ref{crescitaF} we have that for every $\varepsilon>0$ there exists $c_\varepsilon>0$ such that
\begin{equation*}
\max\{|\partial_1 F(\eta,\zeta)|, |\partial_2 F(\eta,\zeta)|\}\leq \varepsilon \left(|\eta|+ |\zeta|\right) + c_\varepsilon \left(|\eta|^{\alpha-1} + |\zeta|^{(\alpha-1)} \right),
\end{equation*}
and
\begin{equation}\label{stimamoduloF}
\begin{split}
|F(\eta,\zeta)|&\leq \varepsilon \left(|\eta|^{2} + 2|\eta||\zeta| + |\zeta|^2\right)\\
&\;\;\;  + c_\varepsilon \left(|\eta|^{\alpha} + |\zeta|^\alpha+|\zeta|^{(\alpha-1)}|\eta|  + |\zeta||\eta|^{(\alpha-1)} \right)
\end{split}
\end{equation}
for any $(\eta,\zeta)\in\R^2$.

Integrating \eqref{stimamoduloF} and using the Young inequality we easily get
\begin{align*}
\int_{\Omega_\psi} K(q)|F(u(q),v(q))| d\mu(q) & \leq
 2\varepsilon\|K\|_\infty\int_{\Omega_\psi} \left(|u(q)|^2 + |v(q)|^2 \right) d\mu(q)\\
& \;\;\; + 2c_\varepsilon\|K\|_\infty \int_{\Omega_\psi} \left(|u(q)|^\alpha +|v(q)|^\alpha \right) d\mu(q),
\end{align*}
for every $(u,v)\in Y_T$.

Furthermore, invoking the embeddings result stated in Lemma \ref{lem1}, since $\nu\in [0,C_V^{-1}n^2)$, we deduce that
\begin{align*}
\int_{\Omega_\psi} K(q)|F(u(q),v(q))| d\mu(q) &\leq \varepsilon c_2\|K\|_\infty \left(\left\|u \right\|^2+\left\|v \right\|^2 \right)\\
&\quad + c_\varepsilon c_\alpha^{\alpha}\|K\|_\infty\left(\left\|u \right\|^\alpha +\left\|v \right\|^\alpha \right).
\end{align*}
\indent Now, taking into account that the real function defined by $\xi\mapsto (a^\xi +b^\xi)^{1/\xi}$, $\xi>0$, $a,b\geq 0$, is nonincreasing, it follows that
$$
\left\| u\right\|^\alpha + \left\| v\right\|^\alpha \leq \left(\left\| u\right\|^2 + \left\| v\right\|^2 \right)^{\alpha/2}
$$
and therefore
\begin{align*}
\mathfrak F(u,v)
& \leq 4\varepsilon c_2\|K\|_\infty\left(\frac{\left\|u \right\|^2 +\left\|v \right\|^2}{2}\right) \\
&\;\;\; + 2^{\alpha/2+1}c_\varepsilon c_\alpha^{\alpha}  \|K\|_\infty\left(\frac{\left\|u \right\|^2 + \left\|v \right\|^2 }{2}\right)^{\alpha/2}.
\end{align*}

The above inequality yields
\begin{align}
f(\xi)&=\sup_{(u,v)\in \Phi^{-1}((-\infty,\xi])}\mathfrak F(u,v)\nonumber \\
& = \sup_{(u,v)\in \Phi^{-1}((-\infty,\xi])} \int_{\Omega_\psi}K(q)F(u(q),v(q))d\mu(q) \nonumber\\
& \leq \|K\|_\infty\sup_{(u,v)\in \Phi^{-1}((-\infty,\xi])} \left(4\varepsilon c_2\left(\frac{\left\|u \right\|^2 +\left\|v \right\|^2}{2}\right)+2^{\alpha/2+1}c_\varepsilon c_\alpha^{\alpha}  \left(\frac{\left\|u \right\|^2 + \left\|v \right\|^2 }{2}\right)^{\alpha/2}\right)\nonumber\\
&\leq \|K\|_\infty\left(4c_2\xi + 2^{\alpha/2+1}c_\varepsilon c_\alpha^{\alpha} \xi^{\alpha/2}\right)\nonumber
\end{align}

\noindent for every $\xi>0$.

Since the nonlinearity $f$ is nonnegative, it follows that
\begin{equation}\label{limitefa0}
\lim_{\xi\to 0^+}\frac{f(\xi)}{\xi}=0.
\end{equation}

\noindent Now, we claim that that there exists $(u_0,v_0)\in Y_T$ such that
 \begin{equation}\label{Fpos}
 \int_{\Omega_\psi}K(\sigma)F(u_0(q),v_0(q))d\mu(q)>0.
\end{equation}

Indeed, following Balogh and Krist\'{a}ly in \cite{bk}, we construct a special test function belonging to $HW^{1,2}_{0,T}(\Omega_\psi)$ that will be useful for our purposes. Let $$\widehat{\Omega}_0=\bigcup_{\widehat{\tau}\in \mathbb{U}(n)}\{\widehat{\tau}*\Omega_0\},$$
where $\Omega_0$ is the open set of $\mathbb H^n$ given in $(h_K)$. Since $K$ is cylindrically symmetric, one has
\begin{equation}\label{cpos}
\inf_{q\in \Omega_0} K(q)=\inf_{q\in \widehat{\Omega}_0} K(q)>0.
\end{equation}

Furthermore, we can find $(z_0,t_0)\in \Omega_\psi$ and
 \begin{equation}\label{Rcpos}
0<R<2|z_0|(\sqrt{2}-1),
\end{equation}
such that

\begin{equation}\label{A_R}
A_R=\{q\in \mathbb H^n: q=(z,t)\,\mbox{with}\, ||z|-|z_0||\leq R,\,|t-t_0|\leq R\}\subset \Omega_0.
\end{equation}

Of course, for every $\varrho\in (0,1]$, it follows that
$$
A_{\varrho R}\subseteq A_R\subset \Omega_0,
$$
and $\mu(A_{\varrho R})>0$.

Set $\varrho\in (0,1)$ and $c_0\in \R$. Let us consider the function $v_\varrho^{c_0}\in HW^{1,2}_{0, \rm cyl}(\Omega_\psi)\subseteq HW^{1,2}_{0,T}(\Omega_\psi)$ given by
\[
v_{\varrho}^{c_0}(q)= \frac{c_0}{1-\varrho}\left(1-\max\left(\frac{||z|-|z_0||}{R},\frac{||t|-|t_0||}{R},\varrho\right)\right)_+,\quad q=(z,t)\in \Omega_\psi
\]
where $\ell_+:=\max\{0,\ell\}$.
With the above notation, we have:
 \begin{itemize}
 \item[$i_1)$] $\textrm{supp}(v_\varrho^{c_0})=A_R$;
 \item[$i_2)$] $\|v_\varrho^{c_0}\|_\infty\leq |c_0|$;
 \item[$i_3)$] $v_\varrho^{c_0}(q)=c_0$ for every $q\in A_{\varrho R}$.
 \end{itemize}

 \indent By $(f_1)$, there exists $(\eta_0,\zeta_0)\in \R^2\setminus\{(0,0)\}$ such that $F(\eta_0,\zeta_0)>0$. Moreover,

\begin{align*}
\int_{\Omega_\psi}
K(q) F(v_\varrho^{\eta_0}(q), v_\varrho^{\zeta_0}(q))\,d\mu(q) & = \int_{A_{\varrho R}}
K(q)F(v_\varrho^{\eta_0}(q), v_\varrho^{\zeta_0}(q))\,d\mu(q)\\
& \quad+\int_{A_{R}\setminus A_{\varrho R}}
K(q)F(v_\varrho^{\eta_0}(q), v_\varrho^{\zeta_0}(q))\,d\mu(q).
\end{align*}

It then follows that
\begin{align}\label{denomi}
\int_{\Omega_\psi}
K(q) F(v_\varrho^{\eta_0}(q), v_\varrho^{\zeta_0}(q))\,d\mu(q) & \geq \inf_{q\in A_{R}}K(q)F(\eta_0,\zeta_0)\mu(A_{\varrho R})\\
& \quad -\|K\|_{\infty}\max_{(|\eta|,|\zeta|)\in[0,|\eta_0|]\times [0,|\zeta_0|]}
 |F(\eta,\zeta)|\mu(A_R\setminus A_{\varrho R}). \nonumber
\end{align}
Since $\mu(A_R\setminus A_{\varrho R})\rightarrow 0$, as $\varrho\to 1^-$, we of course, get
\[
\|K\|_{\infty}\max_{(|\eta|,|\zeta|)\in[0,|\eta_0|]\times [0,|\zeta_0|]}
 |F(\eta,\zeta)|\mu(A_R\setminus A_{\varrho R})\rightarrow 0,
\]
as $\varrho\to 1^-$. Moreover,
\begin{equation*}
\mu(A_{\varrho R})\rightarrow \mu(A_R)
\end{equation*}
as $\varrho\to 1^-$. Thus there exists $\varrho_0>0$ such that
\[
\inf_{q\in A_{R}}K(q)F(\eta_0,\zeta_0)\mu(A_{\varrho_0 R})>\|K\|_{\infty}\max_{(|\eta|,|\zeta|)\in[0,|\eta_0|]\times [0,|\zeta_0|]}
 |F(\eta,\zeta)|\mu(A_R\setminus A_{\varrho_0 R}).
\]

Hence \eqref{Fpos} can be proved by choosing
$$
u_0(q)=v_{\varrho_0}^{\eta_0}(q)= \frac{\eta_0}{1-\varrho_0}\left(1-\max\left(\frac{||z|-|z_0||}{R},\frac{||t|-|t_0||}{R},\varrho_0\right)\right)_+,
$$
and
$$
v_0(q)=v_{\varrho_0}^{\zeta_0}(q)= \frac{\zeta_0}{1-\varrho_0}\left(1-\max\left(\frac{||z|-|z_0||}{R},\frac{||t|-|t_0||}{R},\varrho_0\right)\right)_+,
$$
for every $q=(z,t)\in \Omega_\psi$.

Now, fix $\eta\in\R$ such that
$$
0<\eta < \mathfrak F(u_0,v_0)\left(\frac{\left\|u_0 \right\|^2 + \left\|v_0 \right\|^2}{2}\right)^{-1}.
$$
By \eqref{limitefa0}, there exists $\displaystyle\xi_0\in\left(0, \frac{\left\|u_0 \right\|^2 + \left\|v_0 \right\|^2}{2} \right)$ such that $f(\xi_0)<\eta\xi_0$.

Let $\rho_0>0$ such that
\begin{equation}\label{condrho}
f(\xi_0) < \rho_0 < \xi_0 \mathfrak F(u_0,v_0)\left(\frac{\left\|u_0 \right\|^2 + \left\|v_0 \right\|^2}{2}\right)^{-1}.
\end{equation}
Due to the choice of $\xi_0$, one has $\rho_0 < \mathfrak F(u_0,v_0)$.

 We are  now in  position to prove that the following strict inequality holds
$$\displaystyle\sup_{\lambda\in D}\inf_{(u,v)\in Y_T}\Psi((u,v),\lambda)<\displaystyle\inf_{(u,v)\in Y_T}\sup_{\lambda\in D}\Psi((u,v),\lambda),$$
\noindent i.e. that condition $(\Psi_3)$ of Theorem \ref{teorricceri} is satisfied.

Indeed, the real function $$\lambda\mapsto \inf_{(u,v)\in Y_T}\Psi((u,v),\lambda)$$ is upper semicontinuous on $D$ and
$$
\lim_{\lambda \to \infty} \inf_{(u,v)\in Y_T}\Psi ((u,v),\lambda) \leq \lim_{\lambda \to \infty} \Psi((u_0, v_0), \lambda) = -\infty.
$$

Consequently (see \cite[Chapter I]{MaWi}), there exists $\bar\lambda \in D$ such that
\begin{equation}\label{relazionelambdabar}
\sup_{\lambda\in D}\inf_{(u,v)\in Y_T}\Psi((u,v),\lambda) = \inf_{(u,v)\in Y_T}\Psi((u,v),\bar\lambda).
\end{equation}
For each $(u,v)\in \Phi^{-1}((-\infty,\xi_0])$ we have
$$
\mathfrak F(u,v) \leq f(\xi_0) <\rho_0
$$
and hence
\begin{equation}\label{stimaxi0}
\xi_0 \leq \inf\left\lbrace \Phi(u,v): \mathfrak F(u,v)\geq \rho_0 \right\rbrace.
\end{equation}

On the other hand, we also have
\begin{align*}
\inf_{(u,v)\in Y_T}\sup_{\lambda\in D} \Psi((u,v),\lambda) & = \inf_{(u,v)\in Y_T}\left( \Phi(u,v) + \sup_{\lambda\in D} \left(  \lambda\left( \rho_0 - \mathfrak F(u,v)\right) \right) \right) \\
& = \inf_{(u,v)\in Y_T} \left\lbrace \Phi(u,v): \mathfrak F(u,v)\geq \rho_0\right\rbrace,
\end{align*}
and therefore
\begin{equation}\label{altrastimaxi0}
\xi_0 \leq \inf_{(u,v)\in Y_T}\sup_{\lambda\in D} \Psi((u,v),\lambda).
\end{equation}

\noindent There are two distinct cases.\par

If $0\leq \bar\lambda < \xi_0/\rho_0$, it follows that
$$
\inf_{(u,v)\in Y_T} \Psi((u,v),\bar\lambda) \leq \Phi(0,0) - \bar\lambda \mathfrak F(0,0) +\bar\lambda\rho_0= \bar\lambda\rho_0 < \xi_0,
$$
and inequality $(\Psi_3)$ is verified.

If $\bar\lambda \in (\xi_0/\rho_0,\infty)$ it is easy to note that
$$
\inf_{(u,v)\in Y_T} \Psi((u,v),\bar\lambda)  \leq \Psi((u_0,v_0),\bar\lambda) \leq \Psi((u_0,v_0),\xi_0/\rho_0) < \xi_0.
$$
Hence, also in this case inequality $(\Psi_3)$ is satisfied.

Therefore, fixing $\zeta >\displaystyle\sup_{\lambda \in D}\inf_{(u,v)\in Y_T}\Psi((u,v),\lambda)$, Theorem \ref{teorricceri} assures the existence of a nonempty open set $\Lambda \subseteq D$ with the following property:\par
\smallskip
{\it If $\lambda \in \Lambda$ and $G: \Omega_\psi\times \R^2 \to \R$ is continuous and satisfies $(g_1)$--$(g_3)$, then there exists $\delta >0$ such that, for each $\mu \in (0,\delta)$, the functional
$$
\mathcal E_{\lambda,\mu}(u,v)=\Psi((u,v),\lambda)+\mu \Theta(u,v),\quad \mbox{for every}\,\,(u,v)\in Y_T
$$
has at least two local minima in $$\{(u,v)\in Y_T: \Psi((u,v),\lambda) <\zeta\},$$ say $(u^{(j)}_{\lambda,\mu},v^{(j)}_{\lambda,\mu})$, with $j \in \{1,2\}$.}\par
\smallskip
  Here $\Theta:Y_T\to\R$ is the functional defined by
$$
\Theta(u,v)=-\int_{\Omega_\psi}G(q,u(q),v(q)) d\mu(q).
$$
\indent Notice that, similarly to $\mathfrak F$, the functional $\Theta$ is sequentially weakly continuous on $Y_T$ thanks to assumptions $(g_1)$--$(g_3)$.
 Now $K(\cdot)$ and $G(\cdot, \eta,\zeta)$ are symmetric functions (respectively by $(h_K)$ and $(g_4)$), and the action $\pi_{\sharp}:T\times X\rightarrow X$ given by
$$
\pi_{\sharp}(\widehat{\tau}, (u,v))=(\widehat{\tau}\sharp u, \widehat{\tau}\sharp v),
$$
for every $\widehat{\tau}\in T$ and $(u,v)\in X$, is isometric. Thus, the functional $I_{\lambda,\mu}:X\rightarrow \R$
$$
I_{\lambda,\mu}(u,v)=\frac{1}{2}\left(\|u\|^2+\|v\|^2\right)-\lambda\int_{\Omega_\psi}\!\!K(q)F(u(q),v(q))d\mu(q)
$$
$$
\qquad-\mu\int_{\Omega_\psi} G(q,u(q),v(q)) d\mu(q),
$$
for every $(u,v)\in X$, is $T$--invariant, i.e.
$$
I_{\lambda,\mu}(\pi_{\sharp}(\widehat{\tau}, (u,v)))=I_{\lambda,\mu}(u,v),
$$
for every $(u,v)\in X$,
see Section \ref{sottosezione} for details.

Moreover, $$I_{\lambda,\mu}|_{Y_T}(u,v)=\mathcal E_{\lambda,\mu}(u,v)-\lambda\rho_0=\Psi((u,v),\lambda)+\mu \Theta(u,v)-\lambda\rho_0,$$
for every $(u,v)\in Y_T.$

 By Theorem \ref{lem3},
 $(u^{(j)}_{\lambda,\mu},v^{(j)}_{\lambda,\mu})\in Y_T$, with $j \in \{1,2\}$, turn out
also to be critical points of $I_{\lambda,\mu}$ and hence weak solutions to \eqref{p2}.

Finally, to estimate the norm of $(u^{(j)}_{\lambda,\mu},v^{(j)}_{\lambda,\mu})\in Y_T$, with $j \in \{1,2\}$, we take a nondegenerate compact interval $[a,b] \subset \Lambda$. Notice that one has
\begin{align*}
&\bigcup_{\lambda \in [a,b]}\{(u,v) \in Y_T: \Psi((u,v), \lambda) \leq \zeta \}\\
&\subseteq \{(u,v) \in Y_T: \Psi((u,v), a) \leq \zeta \} \cup \{(u,v) \in Y_T: \Psi((u,v),b) \leq \zeta \}
\end{align*}
and hence the set
$$
S:=\bigcup_{\lambda \in [a,b]}\{(u,v) \in Y_T: \Psi((u,v), \lambda) \leq \zeta \}
$$
is bounded. In conclusion, the local minima of the energy functional $\mathcal E_{\lambda, \mu}$ (defined on $Y_T$) have norm at most equal to $\sigma=\displaystyle\sup_{(u,v) \in S}\|(u,v)\|$. This concludes the proof.
\end{proof}

A direct application of Theorem \ref{th1} is given below.

\begin{ex}\label{esempio}\rm{
Let $\Omega_\psi\subset \mathbb H^n$ be as in \eqref{omega}, with $O=(0,0)\in \Omega_\psi$ and let $V,K:\Omega_\psi\rightarrow \R$ be potentials satisfying respectively $(h_V)$ and $(h_K)$.
Furthermore, let us fix $\alpha\in (2,2^*)$ and let $F:\R^2\to\R$ be a $C^1$--function defined by
	\begin{align*}
	F(\eta,\zeta)= \sin\left(|\eta|^{\alpha} + |\zeta|^{\alpha}\right),
	\end{align*}
	for every $(\eta,\zeta)\in\R^2$.
Then, if $\nu\in [0,C_V^{-1}n^2)$, there exist by Theorem \ref{th1} a number $\sigma>0$ and a nonempty open set $\Lambda\subset (0,\infty)$ such that, for every $\lambda\in \Lambda$, the following singular subelliptic system
\begin{equation*}\left\{\begin{aligned}
&-\Delta_{\mathbb H^n}u-\nu V(q)u+u= \alpha\lambda K(q)|u|^{\alpha-2}u\cos(|u|^{\alpha}+|v|^{\alpha})\quad{\rm in}\,\,\Omega_\psi\\
&-\Delta_{\mathbb H^n}v-\nu V(q)v+v= \alpha\lambda K(q)|v|^{\alpha-2}v\cos(|u|^{\alpha}+|v|^{\alpha})\, \quad{\rm in}\,\,\Omega_\psi\\
& \,\,\,u=v=0 \,\,{\rm on }\,\, \partial\Omega_\psi,
\end{aligned}\right.\end{equation*}
\noindent has at least two weak solutions $(u_{\lambda,\mu}^{(j)}, v_{\lambda,\mu}^{(j)})\in HW^{1,2}_{0,{\rm cyl}}(\Omega_\psi)\times HW^{1,2}_{0,{\rm cyl}}(\Omega_\psi)$, with $j\in \{1,2\}$, lying in the ball
$$
\{(u,v)\in HW^{1,2}_{0,{\rm cyl}
}(\Omega_\psi)\times HW^{1,2}_{0,{\rm cyl}
}(\Omega_\psi):\|(u,v)\|\leq \sigma\}.
$$
In other words,
$$
\left(\|u_{\lambda,\mu}^{(j)}\|_{HW^{1,2}_0(\Omega_\psi)}^2-\nu\int_{\Omega_\psi}V(q)|u_{\lambda,\mu}^{(j)}(q)|^2d\mu(q)\right)^{1/2}\leq \sigma,
$$
and
$$
\left(\|v_{\lambda,\mu}^{(j)}\|_{HW^{1,2}_0(\Omega_\psi)}^2-\nu\int_{\Omega_\psi}V(q)|v_{\lambda,\mu}^{(j)}(q)|^2d\mu(q)\right)^{1/2}\leq \sigma,
$$
for $j\in \{1,2\}$.
}
\end{ex}
\begin{remark}\rm{For the sake of completeness we point out that the results presented in this paper could be also investigated
for a larger class of elliptic equations where the leading term is governed by
some differential operators such as the ones considered in \cite{Mingio1,Mingio3,Mingio2}. However, in these cases some different technical approaches need to be adopted in order to get analogous existence results for this wider class of energies. We will consider these interesting cases in our future investigations.}
\end{remark}
\textbf{Acknowledgments.}
The authors were partially supported
by the Italian MIUR project
{\em Variational methods, with applications to problems in mathematical
physics and geometry} (2015KB9WPT\_009) and the Slovenian Research Agency grants P1-0292 and J1-8131.

\bigskip

\end{document}